\documentclass[11pt]{article}
\usepackage[T1]{fontenc}
\usepackage[utf8]{inputenc}
\usepackage[margin=1in]{geometry}
\usepackage{lmodern}
\usepackage{amsmath,amssymb,amsthm,mathtools,bm}
\usepackage{booktabs,array,tabularx,microtype,enumitem}
\usepackage{algorithm,algpseudocode}
\usepackage[authoryear,round]{natbib}
\usepackage{xcolor}
\usepackage{hyperref}
\usepackage[nameinlink,noabbrev]{cleveref}
\hypersetup{hidelinks,pdftitle={Matching Upper and Lower Bounds for Higher-Order Nonconvex Finite-Sum Optimization},pdfsubject={Fixed-order minimax component-oracle complexity}}
\setlist{leftmargin=*,itemsep=2pt,topsep=4pt}
\allowdisplaybreaks[2]
\newcommand{\R}{\mathbb R}
\newcommand{\E}{\mathbb E}
\newcommand{\Prob}{\mathbb P}
\newcommand{\eps}{\varepsilon}
\newcommand{\norm}[1]{\left\lVert #1\right\rVert}
\newcommand{\opnorm}[1]{\left\lVert #1\right\rVert_{\rm op}}
\newcommand{\ip}[2]{\langle #1,#2\rangle}
\newcommand{\ind}{\mathbf 1}
\newcommand{\K}{K_p}
\newcommand{\Ms}{\mathfrak M_p^{\rm ms}}
\newcommand{\Mi}{\mathfrak M_p^{\rm ind}}
\newcommand{\Lip}{\operatorname{Lip}}

\newcommand{\cJ}{\mathcal J}

\providecommand{\doi}[1]{\href{https://doi.org/#1}{doi: \nolinkurl{#1}}}
\newtheorem{theorem}{Theorem}[section]
\newtheorem{lemma}[theorem]{Lemma}
\newtheorem{proposition}[theorem]{Proposition}
\newtheorem{corollary}[theorem]{Corollary}
\theoremstyle{definition}
\newtheorem{definition}[theorem]{Definition}
\theoremstyle{remark}
\newtheorem{remark}[theorem]{Remark}
\crefname{theorem}{Theorem}{Theorems}
\crefname{lemma}{Lemma}{Lemmas}
\crefname{proposition}{Proposition}{Propositions}
\crefname{corollary}{Corollary}{Corollaries}
\crefname{algorithm}{Algorithm}{Algorithms}
\crefname{section}{Section}{Sections}
\title{\textbf{Matching Upper and Lower Bounds for Higher-Order Nonconvex Finite-Sum Optimization}}
\newcommand{\authornames}{Haihan Zhang, Wendao Wu, Chenheng Zhang, Yanyi Li, Chunyuan Zheng, Cong Fang, Haoxuan Li, Zhouchen Lin}
\newcommand{\affiliation}{Peking University}
\newcommand{\contactemail}{colombo.2026@outlook.com}
\newcommand{\mail}[1]{\href{mailto:#1}{\texttt{#1}}}
\newif\ifanonymous
\anonymousfalse
\ifanonymous
\author{Anonymous Authors}
\hypersetup{pdfauthor={Anonymous Authors}}
\else
\author{Wendao Wu$^{1,*}$\quad Haihan Zhang$^{1,*}$\quad Chenheng Zhang$^{1,*}$\\[0.20em]
Yanyi Li$^1$\quad Chunyuan Zheng$^1$\\[0.20em]
Cong Fang$^{1,\dagger}$\quad Haoxuan Li$^{1,\dagger}$\quad Zhouchen Lin$^{1,\dagger}$\\[0.60em]
\small $^1$\affiliation\\[0.35em]
\footnotesize\mail{\contactemail}\\[0.35em]
\small $^*$Equal contribution.\quad $^\dagger$Corresponding authors.}
\hypersetup{pdfauthor={\authornames}}
\fi
\newcommand{\ResearchAgentSystem}{our laboratory's internal auto-research system}
\date{}
\begin{document}
\maketitle
\begin{abstract}
We establish tight randomized higher-order oracle complexity for finding first-order stationary points of nonconvex finite sums. Let $n$ be the number of components, $\Delta>0$ the initial objective-gap bound, $L_p>0$ an individual $p$th-derivative Lipschitz bound, and $\eps>0$ the target gradient norm. For every fixed integer $p\ge2$, the minimax number of exact component queries returning the value and all derivatives through order $p$, with success probability at least $2/3$, is
\[
 \Theta_p\!\left(n+\Delta L_p^{1/p}n^{1-1/(2p)}\eps^{-(p+1)/p}\right),
\]
where the constants depend only on $p$ and the worst case ranges over all finite dimensions. The lower bound holds for unrestricted randomized adaptive algorithms and closes the $\sqrt n$ gap between the previously known general-order upper and lower bounds in their dependence on $n$. We extend dense weak hiding to complete higher-order replies while keeping each component's regularity independent of the chain length. The matching upper bound retains the known finite-sum exponent, requires only mean-squared $p$th-derivative increments, and removes the fixed-confidence logarithmic loss by verifying entire recursive-estimation epochs with exact function values. The characterization includes the additive $n$ term for every positive parameter regime; it counts oracle calls with unrestricted internal computation.
\end{abstract}

\medskip
{\small
\begin{samepage}
\noindent\textbf{AI Usage.}
Nearly the entire research pipeline for this paper was carried out by
\ResearchAgentSystem{}, powered by GPT-5.6 Sol. The system also conducted a
Lean-backed article audit of the resulting manuscript. The authors subsequently reviewed and approved the
mathematical claims, presentation, and formal artifacts, and take
responsibility for the final manuscript. The complete Lean audit report and the system's technical report will be made public at a later date.
\par
\end{samepage}
}

\paragraph{Publication Intent and Recognition of Independent Work.}

This preprint is released as a research output of our study of
automated scientific research, for public scrutiny and discussion.
As we do not intend to submit this manuscript to a journal or
conference, it has not undergone the full editorial refinement
and internal review we would normally require for submission.
The authors remain responsible for its contents and for correcting
any errors or inaccurate attribution.

We recognize that releasing these results may affect researchers
who have independently been pursuing the same questions, including
work that is not yet public.
If this applies to your work, we invite you to contact us at
\texttt{colombo.2026@outlook.com}.
A brief description of the overlap, together with any dated
materials you are comfortable sharing, would help us understand
its scope and chronology.
Such materials might include draft excerpts, presentation slides,
or revision records; a complete manuscript or unpublished proof
is not required to begin the conversation.

We will review these cases carefully and update citations,
acknowledgments, and accounts of independent or concurrent work
as appropriate, including revising priority claims where warranted.
Nonpublic materials will be treated confidentially and used only
to assess the overlap unless their authors authorize otherwise.
We will obtain explicit permission from the authors before
publicly discussing such materials or providing them to an
automated research system.
\section{Introduction}\label{sec:intro}
We study the number of higher-order component queries needed to find an $\eps$-stationary point of a nonconvex finite sum
\[
 F(x)=\frac1n\sum_{i=1}^n f_i(x),\qquad x\in\R^d,
 \qquad \|\nabla F(x)\|\le\eps.
\]
Fix an integer $p\ge2$. Each query chooses a component and a point and returns its exact value and all derivatives through order $p$. We assume that each $p$th derivative is $L_p$-Lipschitz and that the initial objective gap is at most $\Delta$, with $\Delta,L_p,\eps>0$. Algorithms may choose both indices and points adaptively and use private randomness. Our question is how the component count $n$ changes the best achievable query complexity under this information model.

For a full objective, regularized Taylor methods attain the accuracy exponent $(p+1)/p$ \citep{birgin2017}, and matching lower bounds cover arbitrary randomized local-oracle algorithms \citep{carmon2020}. Computing a full jet by querying every component at every iteration gives no savings from the finite-sum structure. Taylor control variates provide such savings: \citet{zhougu2022} obtain a dimension-free upper bound with component-count exponent $1-1/(2p)$. Yet the randomized lower bound of \citet{emmenegger2022} has exponent $(p-1)/(2p)$. Writing $K_p=\Delta L_p^{1/p}\eps^{-(p+1)/p}$, the previous general-order comparison is
\begin{equation}\label{eq:prior-gap}
 \Omega_p\!\left(K_p n^{(p-1)/(2p)}\right)
 \qquad\text{versus}\qquad
 \widetilde O_p\!\left(K_p n^{1-1/(2p)}\right).
\end{equation}
Here $n\le c_pK_p^{2p/(p+1)}$ is the earlier lower bound's accuracy range, and the tilde hides confidence-related logarithms. Constants with subscript $p$ depend only on the fixed order. The two powers of $n$ differ by exactly $\sqrt n$. It remains to determine whether adaptive component access can improve this upper exponent, or whether a stronger lower bound makes it necessary.

We prove that the upper exponent is necessary. Let $\mathfrak M_p$ denote the smallest deterministic query cap that guarantees an $\eps$-stationary output with probability at least $2/3$ for every admissible instance, taking the worst case over all finite dimensions. Then
\begin{equation}\label{eq:intro-rate}
 \mathfrak M_p
 =\Theta_p\!\left(n+K_p n^{1-1/(2p)}\right).
\end{equation}
The characterization holds for every $n\ge1$ and all positive $\Delta,L_p,\eps$. The lower bound permits arbitrary query points and outputs, without a span or zero-respecting restriction. The upper bound holds on the larger class defined by a mean-squared increment bound on the $p$th derivatives. Consequently, both regularity classes have the same minimax complexity (\Cref{thm:main}). The main extension is to fixed orders $p\ge3$, with $p=2$ included.

\paragraph{Why complete higher-order replies remain hard.}
The main obstacle is to conceal a direction from every tensor in a query reply while keeping each component individually smooth. We extend the dense weak hiding construction of \citet{peng2026} and its exact-Hessian development in the unpublished companion \citep{companion2026}. A hidden direction is spread across a fixed table of weakly biased component rows. Its exact row average retains the direction, but learning it requires many distinct rows. An exactly flat gate conceals every derivative through order $p$ until the preceding stage opens. The key analytic estimate bounds all mixed $(p+1)$st derivatives without a chain-length factor and with only the inverse first power of the hiding bias. A censored transcript then bounds the information available before the first revealing query, including for adaptive indices and unbounded query points. These two estimates yield the missing $\sqrt n$ factor in the lower bound; \Cref{sec:lower-main} explains their balance.

\paragraph{Matching the rate under second moments.}
The upper method uses the Taylor control variate of \citet{zhougu2022}; its finite-sum exponent is inherited. Our analysis instead controls an entire recursive-estimation epoch with constant conditional probability under second moments. An exact evaluation at the endpoint certifies either stationarity or sufficient decrease, so an inaccurate epoch can be rejected while preserving all accepted progress. Amortizing these full evaluations over $\Theta(\sqrt n)$ steps removes the logarithmic loss at fixed confidence. For failure probability $\delta\in(0,1)$, \Cref{thm:refinements} gives the additional cost $O_p(n\log(1/\delta))$. Values, initialization, rejected trials, and verification are included in the query count. As in the underlying higher-order oracle model, global minimization of a returned Taylor model is permitted internal computation.

\begin{table}[tbp]
\centering\small
\setlength{\tabcolsep}{4pt}
\setlength{\belowcaptionskip}{8pt}
\renewcommand{\arraystretch}{1.2}
\caption{Randomized component-query complexity for first-order stationarity at fixed $p\ge2$. Here $K_p=\Delta L_p^{1/p}\eps^{-(p+1)/p}$. Previous rows display the nontrivial term on the common range $n\le c_pK_p^{2p/(p+1)}$; the tilde hides logarithms. Their upper bound uses derivatives alone and is implementable with our value-inclusive oracle. Our bounds include initialization and hold for all positive parameters, with unrestricted adaptive algorithms and worst-case dimension.}
\label{tab:comparison}
\begin{tabularx}{\linewidth}{@{}>{\raggedright\arraybackslash}p{0.25\linewidth} >{\centering\arraybackslash}p{0.22\linewidth} >{\centering\arraybackslash}p{0.22\linewidth} >{\raggedright\arraybackslash}X@{}}
\toprule
Reference & Upper bound & Lower bound & Regularity \\
\midrule
\citet{emmenegger2022} & --- & $\Omega_p(K_p n^{(p-1)/(2p)})$ & Individual $D^p$ \\
\citet{zhougu2022} & $\widetilde O_p(K_p n^{1-1/(2p)})$ & --- & Individual $D^p$ \\
\midrule
\textbf{This Paper}\newline (\Cref{thm:main}) & \multicolumn{2}{c}{$\boldsymbol{\Theta_p(n+K_p n^{1-1/(2p)})}$} & Individual $D^p$ or mean-squared $D^p$ \\
\bottomrule
\end{tabularx}
\end{table}

The additive $n$ term is necessary even for arbitrarily small positive gap: without a bound on lower-order smoothness, one unknown component can move the stationary region of a quadratic. This complements the chain lower bound and gives full parameter coverage. The exact oracle model and main statements follow in \Cref{sec:model}; the algorithm and lower-bound mechanism are developed in \Cref{sec:algorithm-main,sec:lower-main}, with complete proofs in the appendices.

\section{Related Works}\label{sec:related}
\paragraph{Full-function higher-order methods.}
Cubic regularization establishes the second-order foundation \citep{nesterov2006}; general-order regularized models attain the corresponding $(p+1)/p$ stationary-point exponent \citep{birgin2017}. The lower bounds of \citet{carmon2020} show that unrestricted randomized access to local derivatives cannot improve that exponent on the stated smoothness class. These full-function results determine the accuracy dependence, whereas incremental access introduces the component-count question studied here. \citet{doikov2026} analyze lower-order finite-difference implementations of higher-order methods. Their dimension-dependent cost concerns recovering derivative information for a full objective, a different resource from querying one complete component jet. More recently, \citet{zhou2026firstorder} study deterministic first-order oracle lower bounds under higher-order smoothness. Their oracle returns first-order information for a full objective, whereas our bounds concern randomized incremental access to complete $p$th-order component jets.

\paragraph{Variance reduction and finite-sum higher-order methods.}
SPIDER and PAGE use recursive gradient estimation to exploit mean-squared gradient smoothness \citep{fang2018,li2021}. Variance-reduced cubic and trust-region methods extend estimation ideas to second-order models and second-order stationarity \citep{zhou2019svrc,wang2019svrc,shen2019}; Hessian-only sample counts in these settings differ from the total component-jet cost used here. The general-order Taylor estimator of \citet{zhougu2022} provides our closest algorithmic baseline. Our polynomial-residual recursion is its algebraic reformulation, as shown in \eqref{eq:control-variate-identity}. Their Theorem~5.8 has success probability $1-T\delta_0$ for a per-step failure parameter $\delta_0$; achieving fixed overall confidence introduces a logarithmic dependence on the horizon. Our epoch analysis avoids that simultaneous per-step guarantee. Unified analyses of stochastic cubic methods provide complementary choices of estimators and lazy updates \citep{chayti2024}.

For $p=2$, \citet{pasechnyuk2026} report an $n+\widetilde O(n^{1/2}\eps^{-3/2})$ total oracle bound, suppressing other problem constants, under both mean-squared gradient smoothness and third-moment Hessian smoothness. The additional gradient-smoothness bound restricts a lower-order parameter that is unbounded in our class. Their rate therefore concerns a different minimax problem and is not a uniform upper bound under our higher-order assumption alone.

\paragraph{Randomized finite-sum lower bounds.}
\citet{emmenegger2022} establish higher-order lower bounds for randomized algorithms without a span restriction. Their individual-smoothness result is the quantitative baseline in \eqref{eq:prior-gap}; their third-moment Hessian condition defines a separate, larger class with a stronger second-order lower bound. \citet{peng2026} close the first-order individual-smoothness gap using fixed-table dense weak hiding. Our construction preserves that mechanism and supplies the higher-order derivative estimate and complete-jet transcript reduction. The unpublished companion \citep{companion2026} develops the exact-Hessian flat-gate and censored-revelation argument for $p=2$. Relative to it, the present contribution is the fixed-$p\ge3$ extension and the second-moment, log-free upper analysis. We cite this unpublished manuscript for attribution; no theorem from it is used as a black box, and all required construction and information lemmas are proved here.

\paragraph{Probabilistic models and exact acceptance tests.}
Trust-region and regularization methods can safeguard models that are accurate only with constant probability by testing actual objective decrease \citep{bandeira2014,cartis2018}. Our verification mechanism applies this established principle to an entire finite-sum epoch. Its role is to make a full snapshot and a full endpoint evaluation support $\Theta(\sqrt n)$ recursive steps. A dimension-free second-moment estimate controls the epoch, while exact acceptance tests preserve objective decrease through failed trials. The resulting bound counts all these calls and imposes a deterministic total query cap.

\section{Oracle Model and Main Results}\label{sec:model}
By translation the initial point is $0$. An instance in dimension $d$ consists of $n$ functions $f_i\in C^p(\R^d)$ such that
\begin{equation}\label{eq:model}
 F=\frac1n\sum_{i=1}^n f_i,\qquad
 F(0)-\inf_{x\in\R^d}F(x)\le\Delta<\infty.
\end{equation}
For a $k$-linear tensor $A$, we use
$\|A\|_{\mathrm{op}}=\sup_{\|h_1\|,\ldots,\|h_k\|\le1}|A[h_1,\ldots,h_k]|$.
The \emph{individual} regularity class satisfies
\begin{equation}\label{eq:individual}
 \opnorm{\nabla^p f_i(x)-\nabla^p f_i(y)}\le L_p\norm{x-y}
 \quad\text{for every }i,x,y.
\end{equation}
The \emph{mean-squared} class instead satisfies
\begin{equation}\label{eq:ms}
 \left(\frac1n\sum_{i=1}^n
 \opnorm{\nabla^p f_i(x)-\nabla^p f_i(y)}^2\right)^{1/2}
 \le L_p\norm{x-y}\quad\text{for every }x,y.
\end{equation}
The first condition implies the second with the same $L_p$; for $n\ge2$ the inclusion is strict (Appendix~\ref{sec:strictness}). Neither condition bounds lower-order derivative magnitudes or their offsets across components. Lower boundedness is required only of $F$, although our lower-bound components are themselves bounded below.

\begin{definition}[Exact incremental $p$-jet complexity]
A query selects $i\in[n]$ and $x\in\R^d$ and returns
\[
 \cJ_i^p(x)=\bigl(f_i(x),\nabla f_i(x),\ldots,\nabla^p f_i(x)\bigr).
\]
Each complete reply costs one query. An algorithm's indices, points, stopping decisions, and output are arbitrary measurable functions of previous replies and an independent private random tape. Internal computation is unrestricted. For either regularity class, its minimax complexity is the smallest deterministic query budget that guarantees an output $\widehat x$ with $\norm{\nabla F(\widehat x)}\le\eps$ with probability at least $2/3$, uniformly over all finite $d$ and all admissible instances. Denote the two complexities by $\Mi$ and $\Ms$.
\end{definition}

Set
\[
 K_p=\Delta L_p^{1/p}\eps^{-(p+1)/p},\qquad S_p=n+K_p n^{1-1/(2p)}.
\]
\begin{theorem}[Fixed-order minimax characterization]\label[theorem]{thm:main}
For every fixed integer $p\ge2$, there exist constants $c_p,C_p>0$ depending only on $p$ such that, for all $n\ge1$ and $\Delta,L_p,\eps>0$,
\begin{equation}\label{eq:main}
 c_pS_p\le\Mi\le\Ms\le C_pS_p.
\end{equation}
The lower bound holds for arbitrary randomized adaptive algorithms, including unbounded queries and unqueried outputs. The upper bound has success probability at least $31/32$ with a deterministic budget. If $\Delta=0$, zero queries suffice; if $\Delta>0$ and $L_p=0$, both complexities are $\Theta(n)$.
\end{theorem}

For $p=3$ and $p=4$, the common rates are respectively
\[
 \Theta\!\left(n+\Delta L_3^{1/3}n^{5/6}\eps^{-4/3}\right),\qquad
 \Theta\!\left(n+\Delta L_4^{1/4}n^{7/8}\eps^{-5/4}\right).
\]
The constants are dimension-free, not uniform in an order $p$ that grows with the other parameters.

\begin{corollary}[Higher moment classes]\label[corollary]{cor:moments}
For any fixed or variable $q\in[2,\infty]$, replace \eqref{eq:ms} by the normalized $q$th-moment bound with the same $L_p$ (the maximum when $q=\infty$). Its fixed-confidence minimax complexity is $\Theta_p(S_p)$, with constants independent of $q$.
\end{corollary}
\begin{proof}
Normalized moments are nondecreasing in their order. Every such class contains the individual class and is contained in the mean-squared class. Apply \eqref{eq:main}.
\end{proof}
This includes third-moment Hessian smoothness at $p=2$, but does not assert that second moments are a necessary threshold.

\begin{theorem}[Confidence and lower-bound dimension]\label[theorem]{thm:refinements}
Under \eqref{eq:ms}, for every $\delta\in(0,1)$, success probability $1-\delta$ is achievable with
\begin{equation}\label{eq:confidence-main}
 O_p\!\left(K_p n^{1-1/(2p)}+n[1+\log(1/\delta)]\right)
\end{equation}
queries on every path. The lower bound in \Cref{thm:main} can be witnessed in dimension
\begin{equation}\label{eq:dimension-main}
 d\le C_p(1+K_p n^{-1/(2p)})
       \log(2+n+K_p n^{1-1/(2p)}).
\end{equation}
No matching dependence on $\delta$ is claimed.
\end{theorem}
Proofs of the main and refinement theorems appear in Appendices~\ref{sec:upper}--\ref{sec:scaling}. We next explain the algorithm and the two matching balances.

\section{Verified Taylor-Residual Epochs}\label{sec:algorithm-main}
At a snapshot $z$, one full pass stores the component Taylor polynomials
\[
 P_i(y)=\sum_{j=0}^p\frac1{j!}\nabla^j f_i(z)[(y-z)^{\otimes j}],
 \qquad P=\frac1n\sum_iP_i.
\]
Higher derivatives within an epoch are evaluated from $P$ without further queries. Only the gradient is tracked recursively. After selecting a step $s_t$ and putting $x_{t+1}=x_t+s_t$, use fresh uniform indices in
\begin{equation}\label{eq:recursion-main}
 \begin{split}
 v_{t+1}={}&v_t+\nabla P(x_{t+1})-\nabla P(x_t)\\
 &+\frac1B\sum_{b=1}^B
 \bigl[\nabla f_{I_{t,b}}(x_{t+1})-\nabla f_{I_{t,b}}(x_t)
       -\nabla P_{I_{t,b}}(x_{t+1})+\nabla P_{I_{t,b}}(x_t)\bigr].
 \end{split}
\end{equation}
The sampled residuals use two queries each. Storing complete snapshot jets makes every Taylor correction available at no additional oracle cost.

For all trial epochs choose
\begin{equation}\label{eq:main-parameters}
 \begin{gathered}
 m=\lceil\sqrt n\rceil,\quad B=64m,\quad
 M=4(p^2+2^{p+1})L_pm^{p-1},\\
 r^p=\frac{p!\eps}{4M},\qquad
 d_0=\frac{Mr^{p+1}}{2(p+1)!},\qquad
 H=\left\lfloor\frac{\Delta}{md_0}\right\rfloor+1.
 \end{gathered}
\end{equation}
The step is a global minimizer, on the closed ball of radius $r$, of
\begin{equation}\label{eq:model-main}
 \langle v_t,s\rangle+
 \sum_{j=2}^p\frac1{j!}\nabla^jP(x_t)[s^{\otimes j}]
 +\frac{M}{(p+1)!}\|s\|^{p+1}.
\end{equation}
A measurable minimizer exists. Global model solution is permitted internal computation; this is where the oracle model deliberately differs from an arithmetic-time model.

\begin{algorithm}[t]
\caption{Verified Taylor-residual epochs (fixed confidence)}\label{alg:verified}
\begin{algorithmic}[1]
\Require $n,p,\Delta,L_p,\eps>0$; parameters in \eqref{eq:main-parameters}
\State $z\gets0$
\For{$\ell=1,\ldots,2H$}
  \State Query every component at $z$; form $P_i,P$, $F(z)$, and $v\gets\nabla F(z)$.
  \State $x\gets z$; $\mathsf{full}\gets\mathsf{true}$
  \For{$t=0,\ldots,m-1$}
    \State Select a global minimizer $s$ of \eqref{eq:model-main}; set $y\gets x+s$.
    \If{$\|s\|<r$}
       \State $\mathsf{full}\gets\mathsf{false}$; \textbf{break}
    \EndIf
    \If{$t<m-1$}
       \State Sample $B$ fresh uniform indices and update $v$ by \eqref{eq:recursion-main}.
    \EndIf
    \State $x\gets y$
  \EndFor
  \State Query every component at $y$; compute exact $F(y)$ and $\nabla F(y)$.
  \If{$\|\nabla F(y)\|\le\eps$}
     \State \Return $y$
  \EndIf
  \If{$\mathsf{full}$ and $F(y)\le F(z)-md_0$}
     \State $z\gets y$ \Comment{Accept only verified progress.}
  \EndIf
  \Statex \hspace{\algorithmicindent}\textit{Otherwise keep $z$; the next trial uses fresh samples.}
\EndFor
\State \Return $z$ \Comment{The cap is reached with probability at most $1/32$.}
\end{algorithmic}
\end{algorithm}

\paragraph{Why second moments suffice.}
With steps bounded by $r$ and epoch length $m$, the residual's conditional second moment is at most
\[
 \sigma^2,\qquad \sigma=\frac{L_pm^{p-1}r^p}{(p-1)!}.
\]
The error $e_t=v_t-\nabla F(x_t)$ is a vector martingale starting at zero. Its terminal second moment is at most $m\sigma^2/B$, so a maximal inequality gives
\[
 \Prob\!\left(\max_{t\le m}\|e_t\|>\sigma\mid\text{starting history}\right)
 \le \frac mB=\frac1{64}.
\]
On the complementary event, every boundary step decreases the objective by at least $d_0$, and every interior step has an endpoint with gradient norm at most $\eps/2$. These statements follow from deterministic Taylor error bounds; no random tensor concentration is used.

\paragraph{Why errors do not accumulate across epochs.}
Every output from the stationarity test is exactly certified. Every accepted nonterminal epoch has real objective decrease at least $md_0$, whether or not its estimates were correct. Every good epoch either terminates or is accepted. Hence nontermination after $2H$ trials requires at least $H$ bad trials, although their expected number is at most $2H/64$. Markov's inequality gives failure probability at most $1/32$. This argument uses conditional bounds, not independence of epoch outcomes.

\paragraph{The upper-bound balance.}
Each trial costs at most $2n+2Bm=O(n)$ queries, counting endpoint verification. The number of successful epochs that the objective gap allows is
\[
 1+\frac{\Delta}{md_0}=O_p(1+K_pm^{-1/p}).
\]
Substituting $m\asymp\sqrt n$ gives $O_p(n+K_p n^{1-1/(2p)})$. Appendix~\ref{sec:upper} proves the deterministic query cap and the confidence refinement.

\section{Higher-Order Dense Weak Hiding}\label{sec:lower-main}
The lower bound in \Cref{thm:main} requires three properties: an unfinished stage must force a large gradient, opening a stage must require many distinct component rows, and the entire construction must remain individually $p$th-order smooth. We summarize how these requirements fit together; \Cref{lem:base,prop:sequential} establish the analytic and information bounds, and \Cref{sec:scaling} combines them.

\subsection{A fixed-table radial chain}
For each stage $j\in[T]$, choose an independent uniform hidden vector $\Theta^{(j)}\in\{\pm1\}^D$. Given this vector, sample a fixed sign table $S^{(j)}\in\{\pm1\}^{n\times D}$ whose columns have exact sums $n\rho\Theta^{(j)}_k$. The bias $\rho$ is chosen on the parity-compatible grid. Write
\[
 \Psi(z)=\frac{z}{\sqrt{1+\|z\|^2}},\quad
 q_j(z)=\left\langle\Theta^{(j)}/\sqrt D,\Psi(z)\right\rangle,\quad
 \widetilde q_{ij}(z)=\rho^{-1}\left\langle S^{(j)}_{i:}/\sqrt D,\Psi(z)\right\rangle.
\]
Thus $n^{-1}\sum_i\widetilde q_{ij}=q_j$ exactly, not in expectation.

Let $a=1/8$, $b=1/4$, and let $G_p$ be a nondecreasing gate equal to zero on $(-\infty,a]$, equal to one on $[b,\infty)$, and flat through order $p+1$ at both interfaces. A normalized integral of $t^{p+1}(1-t)^{p+1}$ gives an explicit degree-$(2p+3)$ transition polynomial. Define
\begin{equation}\label{eq:hard-main}
 R_i(x)=-\widetilde q_{i1}(x_1)
 -\sum_{j=2}^T G_p(q_{j-1}(x_{j-1}))\bigl(\widetilde q_{ij}(x_j)+1\bigr)
 +\sum_{j=1}^T\|x_j\|^2.
\end{equation}
The average $R$ replaces each $\widetilde q_{ij}$ by $q_j$. Its three analytical properties are
\begin{equation}\label{eq:hard-properties-main}
 \begin{gathered}
 R(0)-\inf R\le2T,\qquad
 \min_jq_j(x_j)<b\ \Longrightarrow\ \|\nabla R(x)\|>1/3,\\
 \Lip(\nabla^pR_i)\le H_{p+1}/\rho,
 \qquad \log H_{p+1}=O(p\log(p+2)).
 \end{gathered}
\end{equation}
The gradient statement uses the first unfinished block. Its gradient is $2x_j-c\nabla q_j(x_j)$ for some $c\ge1$, whose norm is bounded below uniformly whenever $q_j\le b$.

\subsection{Why the derivative estimate is dimension-free}
All fixed-order derivatives of a radial coordinate have dimension-free bounds. For a link involving blocks $j-1,j$, the product rule gives, with $k=p+1$,
\[
 |D^k[ G_p(q_{j-1})(\widetilde q_{ij}+1)]
 [h^1,\ldots,h^k]|
 \le\frac{C_{p,k}}\rho
 \prod_{\ell=1}^k\bigl(\|h^\ell_{j-1}\|+\|h^\ell_j\|\bigr).
\]
A direct triangle bound would introduce $T$. Instead set $a_j^\ell=\|h^\ell_{j-1}\|+\|h^\ell_j\|$. Finite overlap and $k\ge2$ imply
\[
 \sum_j\prod_{\ell=1}^k a_j^\ell
 \le\prod_{\ell=1}^k\|a^\ell\|_{\ell_k}
 \le2^k\prod_{\ell=1}^k\|h^\ell\|.
\]
Only one weak factor appears in each link, so the bias cost remains $1/\rho$, not a higher power. These two facts give the last line of \eqref{eq:hard-properties-main}.

\subsection{Complete replies, censoring, and adaptive information}
If $q_j(x_j)\le a$, the outgoing link has value and every derivative through order $p$ exactly zero. Therefore, before stage $j$ first opens, its public reply is a deterministic function of completed-stage data and the selected row of table $j$. Revealing that entire row only strengthens the algorithm.

For a stage with at most $r$ distinct rows exposed, the censored row transcript satisfies
\[
 I(\Theta;\mathcal T_r)\le Dr\rho^2.
\]
Under an independent hidden direction, any transcript-measurable vector $V$ with $\|V\|\le1$ has alignment exceeding $a$ with probability at most $e^{-D/128}$. Relative-entropy data processing then shows that a stage is unlikely to open in fewer than
\[
 r_\rho=\left\lfloor\frac1{4096}\min\{n,\rho^{-2}\}\right\rfloor
\]
assigned calls. Crucially, the censored transcript excludes the completion response: the response may reveal the direction, but it arrives only after the first alignment query has been chosen. A first-hit coupling makes this distinction rigorous.

A truncated public simulator reveals at most one new stage per query. Its untouched future stages remain independent of its history, so a union bound controls all unintended future alignments in block dimension $D=O(1+\log((N+1)T))$. On the no-future-alignment event, exact gate flatness couples its public replies to those of the true oracle. Stage costs add without assuming independent durations. Reserving one untouched final stage also handles an arbitrary unqueried output. The resulting lower bound is
\begin{equation}\label{eq:row-stage-main}
 N=\Omega\!\left(T\min\{n,\rho^{-2}\}\right).
\end{equation}
The bounded map $\Psi$ makes this argument valid for every query radius.

\subsection{The matching bias balance and remaining regimes}
Scale by $f_i(y)=\alpha R_i(\beta y)$ with
\[
 \beta^p=\frac{L_p\rho}{6H_{p+1}\eps},\qquad
 \alpha=\frac{6\eps}{\beta}.
\]
This enforces individual smoothness and a gradient barrier above $2\eps$. The gap budget permits
\[
 T=\Theta_p(K_p\rho^{1/p})
\]
whenever the right side is sufficiently large. Consequently \eqref{eq:row-stage-main} becomes
\[
 \Omega_p\!\left(K_p\min\{n\rho^{1/p},\rho^{-2+1/p}\}\right).
\]
The increasing and decreasing branches balance at $\rho\asymp n^{-1/2}$, yielding $\Omega_p(K_pn^{1-1/(2p)})$.

For the complementary regime, a single hidden component contains an unknown linear perturbation of a public quadratic. Until that component is queried, four disjoint possible stationary regions are indistinguishable, giving $\Omega(n)$. Small component counts use identical copies of a single hidden radial chain. Appendix~\ref{sec:scaling} supplies the integer choices, probability constants, dimension bound, and boundary-parameter cases.

\section{Consequences and Limitations}\label{sec:discussion}
The fixed-order finite-sum price is $n^{1-1/(2p)}$: the exponents $3/4,5/6,7/8,\ldots$ are intrinsic to exact incremental access under individual higher-order smoothness. The proof identifies the same exponent from two different balances. The algorithm amortizes full passes over $\sqrt n$-step verified epochs; the hard instance requires order $n$ rows per weakly hidden stage while allowing order $K_pn^{-1/(2p)}$ stages. Moment-class inclusion makes this law unchanged under any normalized moment assumption of order at least two.

This theorem does not compare the computational price of forming tensors of different orders. It fixes the order of the oracle and the regularity class together; it therefore does not assert that richer information is intrinsically harmful on one fixed class. Nor does it provide a sharp second-order-stationarity complexity, a tensor-vector-product lower bound, a necessary second-moment threshold, or optimal dependence on a growing $p$. The high-confidence upper bound need not be minimax in $\delta$. Exact function values and unrestricted model computation are substantive features of the result, not omitted implementation costs.

\bibliographystyle{plainnat}
\bibliography{references}
\clearpage
\appendix
\section{Proof of the Upper Bound}\label{sec:upper}
The proof first controls Taylor residuals over one adaptive epoch, then turns that event into verified progress and a total query cap. Only condition \eqref{eq:ms} is used. Let $L=L_p$. By Jensen's inequality, $\nabla^p F$ is $L$-Lipschitz.

At a snapshot $z$, query every component and store its entire $p$-jet. Define the scalar Taylor polynomials
\begin{equation}\label{eq:poly}
 P_i(y)=\sum_{j=0}^p\frac1{j!}\nabla^j f_i(z)[(y-z)^{\otimes j}],
 \qquad P=\frac1n\sum_iP_i.
\end{equation}
Polynomial evaluations and tensor contractions require no new oracle calls.

\begin{lemma}[Taylor errors and residual second moments]\label[lemma]{lem:taylor}
For $2\le j\le p$,
\begin{equation}\label{eq:taylor-j}
 \opnorm{\nabla^j F(x)-\nabla^jP(x)}
 \le\frac{L\norm{x-z}^{p+1-j}}{(p+1-j)!}.
\end{equation}
For $\norm{x-z}\le(m-1)r$ and $\norm s\le r$, set
\begin{align}
 Y_i(x,s;z)&=\nabla f_i(x+s)-\nabla f_i(x)
       -\nabla P_i(x+s)+\nabla P_i(x),\label{eq:residual}\\
 \sigma&=\frac{Lm^{p-1}r^p}{(p-1)!}.
\end{align}
Then
\begin{equation}\label{eq:residual-second}
 \E_i\norm{Y_i(x,s;z)}^2\le\sigma^2,
\end{equation}
where $i$ is uniform on $[n]$.
\end{lemma}
\begin{proof}
Taylor's formula applied to $\nabla^jF$ gives \eqref{eq:taylor-j}. Applying the same integral remainder formula componentwise and using Minkowski's inequality in the finite normalized $\ell_2$ space gives
\[
 \left(\E_i\opnorm{\nabla^2f_i(y)-\nabla^2P_i(y)}^2\right)^{1/2}
 \le\frac{L\norm{y-z}^{p-1}}{(p-1)!}.
\]
For $p=2$ this is exactly \eqref{eq:ms}. For $p>2$, put $h=y-z$ and use the identity
\[
 \nabla^2 f_i(y)-\nabla^2P_i(y)
 =\frac1{(p-3)!}\int_0^1(1-t)^{p-3}
 [\nabla^p f_i(z+th)-\nabla^p f_i(z)][h^{\otimes(p-2)}]dt.
\]
The normalized $L_2$ norm of the integrand is at most $Lt\|h\|^{p-1}$. The integral of $t(1-t)^{p-3}/(p-3)!$ is $1/(p-1)!$. Tensor contraction leaves a bilinear form in this display. Now
\[
 Y_i(x,s;z)=\int_0^1
 [\nabla^2f_i(x+ts)-\nabla^2P_i(x+ts)]s\,dt.
\]
Every point in this integral is within $mr$ of $z$. A second application of Minkowski proves \eqref{eq:residual-second}.
\end{proof}

An epoch starts at $x_0=z$, $v_0=\nabla F(z)$, and has at most $m$ steps. After an adaptive step $s_t$ has been chosen, use fresh independent uniform indices $I_{t,1},\ldots,I_{t,B}$ and update
\begin{equation}\label{eq:grad-update}
 v_{t+1}=v_t+\nabla P(x_t+s_t)-\nabla P(x_t)
  +\frac1B\sum_{b=1}^B Y_{I_{t,b}}(x_t,s_t;z).
\end{equation}
Only two component queries per sampled residual are needed, since all derivatives at $z$ are stored. Algebraically, this is the Taylor control-variate recursion of \citet[Algorithm~3]{zhougu2022}: for any degree-$p$ polynomial $P_i$,
\begin{equation}\label{eq:control-variate-identity}
 \nabla P_i(x+s)-\nabla P_i(x)
 =\sum_{k=1}^{p-1}\frac{1}{k!}
   \sum_{u=0}^{p-k-1}\frac{1}{u!}
   \nabla^{k+u+1}f_i(z)[(x-z)^{\otimes u},s^{\otimes k},\,\cdot\,].
\end{equation}
The contribution here is the mean-squared analysis and verified-epoch wrapper, not the underlying Taylor control variate.

\begin{lemma}[One-epoch success with a constant batch multiplier]\label[lemma]{lem:epoch-good}
Suppose every step has norm at most $r$ and $B\ge64m$. Then, conditional on the complete history before the epoch,
\begin{equation}\label{eq:epoch-good}
 \Prob\left(\max_{0\le t\le m}\norm{v_t-\nabla F(x_t)}>\sigma\right)
 \le\frac1{64}.
\end{equation}
The statement remains valid for an epoch that stops early.
\end{lemma}
\begin{proof}
Put $e_t=v_t-\nabla F(x_t)$. Conditional on the history and the chosen step, the increment $e_{t+1}-e_t$ has mean zero and conditional squared norm in expectation at most $\sigma^2/B$. Thus $e_t$ is a vector martingale and
\[
 \E\norm{e_m}^2\le\frac{m\sigma^2}{B}.
\]
Pad an early-stopped epoch with zero increments. In \Cref{alg:verified}, only $e_0,\ldots,e_{m-1}$ are needed to choose steps; the unperformed final estimator update is also padded with a zero martingale increment. The nonnegative submartingale $\norm{e_t}^2$ satisfies the maximal inequality
\[
 \Prob\left(\max_t\norm{e_t}>\sigma\right)
 \le\frac{\E\norm{e_m}^2}{\sigma^2}\le\frac1{64}.
\]
All statements hold conditionally on an arbitrary starting history. No coordinatewise tensor concentration, dimension factor, bounded-gradient assumption, or tail envelope is used.
\end{proof}

Set
\begin{equation}\label{eq:upperparams}
 m=\lceil\sqrt n\rceil,\qquad B=64m,\qquad
 A=Lm^{p-1},\qquad M=4(p^2+2^{p+1})A,
 \qquad r^p=\frac{p!\eps}{4M}.
\end{equation}
At point $x_t$, choose a global minimizer over $\norm s\le r$ of
\begin{equation}\label{eq:localmodel}
 q_t(s)=\ip{v_t}{s}
 +\sum_{j=2}^p\frac1{j!}\nabla^jP(x_t)[s^{\otimes j}]
 +\frac{M}{(p+1)!}\norm s^{p+1}.
\end{equation}
A fixed Borel minimizer selector exists; Appendix~\ref{sec:selector} gives a constructive measurability argument. This is a permissible internal computation in the stated oracle model, not a polynomial-time subroutine guarantee.

\begin{lemma}[Deterministic model alternatives]\label[lemma]{lem:model-alt}
On the good event in \eqref{eq:epoch-good}, every model minimizer satisfies
\begin{align}
 \norm{s_t}=r&\quad\Longrightarrow\quad
 F(x_t+s_t)\le F(x_t)-d_0,
 &d_0&=\frac{M r^{p+1}}{2(p+1)!},\label{eq:decrease}\\
 \norm{s_t}<r&\quad\Longrightarrow\quad
 \norm{\nabla F(x_t+s_t)}\le\eps/2.\label{eq:stationary}
\end{align}
\end{lemma}
\begin{proof}
Let $E_1=\norm{v_t-\nabla F(x_t)}$ and
$E_j=\opnorm{\nabla^jP(x_t)-\nabla^jF(x_t)}$ for $j\ge2$.
Throughout the epoch,
\[
 E_1\le\frac{Ar^p}{(p-1)!},\qquad
 E_j\le\frac{Ar^{p+1-j}}{(p+1-j)!}\quad(2\le j\le p).
\]
Here $m^{p+1-j}\le m^{p-1}$ for $j\ge2$. Since $q_t(s_t)\le q_t(0)=0$, on the boundary Taylor's formula gives
\begin{align*}
 F(x_t+s_t)-F(x_t)
 &\le-\frac{M-L}{(p+1)!}r^{p+1}
       +\sum_{j=1}^p\frac{E_jr^j}{j!}\\
 &\le-\frac{M-A(p^2+2^{p+1}-2)}{(p+1)!}r^{p+1}
 \le-\frac{M}{2(p+1)!}r^{p+1}.
\end{align*}
We used the exact identity
\[
 \frac1{(p-1)!}+\sum_{j=2}^p\frac1{j!(p+1-j)!}
 =\frac{p^2+2^{p+1}-3}{(p+1)!}.
\]
In the interior, $\nabla q_t(s_t)=0$. Comparing this equation with the Taylor expansion of $\nabla F(x_t+s_t)$ yields
\begin{align*}
 \norm{\nabla F(x_t+s_t)}
 &\le\frac{Mr^p}{p!}+\frac{Lr^p}{p!}
      +\sum_{j=1}^p\frac{E_jr^{j-1}}{(j-1)!}\\
 &\le\frac{[M+A(p+2^p-1)]r^p}{p!}
 \le\frac{2Mr^p}{p!}=\frac\eps2.
\end{align*}
\end{proof}

\subsection{Exact verification, rejection, and a deterministic total budget}
A \emph{trial epoch} operates as follows. Query all components at its starting snapshot $z$. Run the preceding procedure until either an interior model step occurs or $m$ boundary steps have been taken. Call the resulting candidate $y$. Query all components at $y$ to obtain the exact values $F(y)$ and $\nabla F(y)$.

If $\norm{\nabla F(y)}\le\eps$, return $y$. Otherwise accept $y$ as the next snapshot only if the trial took $m$ boundary steps and
\begin{equation}\label{eq:verification}
 F(y)\le F(z)-m d_0.
\end{equation}
In every other case reject the trial and keep the old snapshot $z$. Rejected trials use fresh samples on the next attempt. Reusing stored jets can save queries, but is unnecessary for the bound.

Every output from the stationarity test is certified. Every accepted nonterminal trial decreases $F$ by at least $m d_0$, even when its gradient estimates were inaccurate. Every good trial either returns a certified point or is accepted, by \Cref{lem:model-alt}.

Let
\begin{equation}\label{eq:trial-budget}
 H=\left\lfloor\frac{\Delta}{m d_0}\right\rfloor+1,
 \qquad R=2H.
\end{equation}
Stop after $R$ trials if no point has been certified and return an arbitrary point in that exceptional case.

\begin{proposition}[No logarithmic loss at constant success]\label[proposition]{prop:upper}
The verified-epoch procedure succeeds with probability at least $31/32$ and uses at most $C_pS_p$ component queries on every execution path.
\end{proposition}
\begin{proof}
Let a bad-trial indicator be one when the good event from \Cref{lem:epoch-good} fails. After termination pad remaining trials by zero bad indicators. Conditional bad probabilities are at most $1/64$, so the expected number of bad trials among the first $R$ is at most $R/64$.

If the algorithm has not terminated, it cannot have encountered $H$ good trials: each would have generated an accepted decrease of at least $m d_0$, contradicting
$H m d_0>\Delta\ge F(0)-\inf F$. Thus nontermination after $R=2H$ implies at least $R/2$ bad trials. Markov's inequality bounds its probability by $1/32$. This does not require independence between trials.

Each trial uses at most $2n+2Bm$ queries, including both its snapshot and endpoint verification. Since $m^2\le4n$ and $B=64m$, this is at most $514n$. Moreover, the definitions of $r,M,d_0$ imply
\begin{equation}\label{eq:epoch-scaling}
 \frac{\Delta}{m d_0}
 =C'_p\K m^{-1/p},\qquad
 C'_p=\frac{2(p+1)!4^{(p+1)/p}[4(p^2+2^{p+1})]^{1/p}}
 {(p!)^{(p+1)/p}}.
\end{equation}
Therefore the deterministic query budget is
\[
 O\!\left(n\left[1+C'_p\K m^{-1/p}\right]\right)
 =O_p\!\left(n+\K n^{1-1/(2p)}\right).
\]
The failure event can only arise from the deterministic cap, not from accepting an uncertified point as stationary.
\end{proof}

\begin{proposition}[Arbitrary confidence]\label[proposition]{prop:confidence}
For every $\delta\in(0,1)$, replace the trial cap by
\[
 R_\delta=2H+\left\lceil4\log(1/\delta)\right\rceil.
\]
The algorithm returns an $\eps$-stationary point with probability at least $1-\delta$, within
\[
 O_p\!\left(\K n^{1-1/(2p)}+n[1+\log(1/\delta)]\right)
\]
queries on every path.
\end{proposition}
\begin{proof}
Let $Z_r$ denote the number of bad trials among the first $r$, using zero indicators after termination. The conditional bad probability bound gives
$\E[2^{Z_r}\mid\text{history before trial }r]\le (65/64)2^{Z_{r-1}}$.
Induction yields $\E 2^{Z_r}\le(65/64)^r$.
For $r\ge2H$, nontermination implies $Z_r\ge r/2$. Therefore
\[
 \Prob(\text{nontermination after }r)
 \le \left(\frac{65}{64\sqrt2}\right)^r
 \le e^{-r/4}.
\]
The prescribed cap is at least $4\log(1/\delta)$. The query count follows from the same per-trial cost as in \Cref{prop:upper}.
\end{proof}
\begin{remark}
This is an upper bound in $\delta$, not a claim of confidence-optimal minimax complexity. Exact component values are used in \eqref{eq:verification}; all corresponding calls are charged. The fixed-confidence theorem does not invoke a union bound over all iterates.
\end{remark}

\section{Explicit Gates and Dimension-Free Derivative Bounds}\label{sec:regularity}
We now construct individually smooth hard instances. Use thresholds
\[
 a=\frac18,\qquad b=\frac14,\qquad w=b-a=\frac18.
\]
For the fixed target order $p$, let
\begin{equation}\label{eq:gate-poly}
 Q_p(t)=Z_p\int_0^t u^{p+1}(1-u)^{p+1}\,du,
 \qquad Z_p=\frac{(2p+3)!}{((p+1)!)^2},
\end{equation}
and set
\begin{equation}\label{eq:gate}
 G_p(s)=\begin{cases}
 0,&s\le a,\\
 Q_p((s-a)/w),&a<s<b,\\
 1,&s\ge b.
 \end{cases}
\end{equation}
This is a nondecreasing $C^{p+1,1}$ function. It is flat through order $p+1$ at both ends: the function and its derivatives through order $p+1$ vanish on $(-\infty,a]$, and $G_p=1$ with all positive-order derivatives through $p+1$ zero on $[b,\infty)$. One extra order of flatness avoids any interface differentiability issue when bounding $D^{p+1}$.

For example, when $p=3$,
\[
 Q_3(t)=126t^5-420t^6+540t^7-315t^8+70t^9.
\]
For $1\le k\le p+1$, an explicit derivative envelope is
\begin{equation}\label{eq:gate-constants}
 \norm{G_p^{(k)}}_\infty\le E_{p,k}
 :=8^k Z_p\sum_{\ell=0}^{p+1}\binom{p+1}{\ell}
 \frac{(p+1+\ell)!}{(p+2+\ell-k)!},
 \qquad E_{p,0}=1.
\end{equation}
This follows by expanding $(1-t)^{p+1}$ and differentiating $Q_p'$; all denominators are factorials of nonnegative integers.

For a unit vector $u\in\R^D$, write
\[
 \Psi(z)=\frac{z}{\sqrt{1+\norm z^2}},\qquad
 q_u(z)=\ip{u}{\Psi(z)}.
\]

\begin{lemma}[All fixed-order radial derivatives]\label[lemma]{lem:radial}
Let $B_0=1$ and $B_k=2\cdot4^k k^k$ for $k\ge1$. For every $D$, every unit $u$, and every $k\ge0$,
\begin{equation}\label{eq:radial}
 \sup_z\opnorm{D^kq_u(z)}\le B_k.
\end{equation}
The order-zero norm means absolute value.
\end{lemma}
\begin{proof}
The order-zero bound follows from $\norm{\Psi(z)}<1$. Fix real $z$ and a real unit direction $h$, put $s=1+\norm z^2$, and consider the holomorphic expression
\[
 t\longmapsto\frac{\ip{u}{z+th}}
 {\sqrt{s+2t\ip{z}{h}+t^2}}.
\]
For $|t|\le\sqrt{s}/4$, the perturbation of the denominator's square is at most $9s/16$. There is an analytic square-root branch on this disk, with denominator modulus at least $\sqrt{7s}/4$. The numerator modulus is at most $5\sqrt{s}/4$. Thus the displayed function has modulus less than two. Cauchy's inequality gives
\[
 |D^kq_u(z)[h,\ldots,h]|\le2\,k!\,4^k.
\]
The real polarization identity for a symmetric $k$-linear form bounds its full multilinear norm by $k^k/k!$ times its diagonal norm. This gives \eqref{eq:radial} for $k\ge1$, independently of $D$.
\end{proof}

We give explicit constants for the component regularity. Let $\Pi_k$ denote the set of partitions of $[k]$, and define
\begin{align}
 A_{p,0}&=1,\qquad
 A_{p,k}=\sum_{\pi\in\Pi_k} E_{p,|\pi|}\prod_{V\in\pi}B_{|V|}
 \quad(1\le k\le p+1),\label{eq:composition-constants}\\
 C_{p,k}&=2\max_{0\le h\le k} A_{p,h}B_{k-h},\qquad
 H_{p+1}=B_{p+1}+2^{p+1}C_{p,p+1}.\label{eq:Hconstant}
\end{align}
These constants depend only on $p$. The partition form of the chain rule gives
$\opnorm{D^k(G_p\circ q_u)}\le A_{p,k}$. The factorial sums and partition counts in these formulas also give $\log H_{p+1}=O(p\log(p+2))$; a derivation is given in Appendix~\ref{sec:constant-growth}. Their growth is not optimized.

\section{The Hard Finite Sum and Its Analytic Properties}\label{sec:hard}
Choose an integer $s$ with the parity of $n$ and $1\le s\le n/8$, and put $\rho=s/n$. At each stage $j\in[T]$, independently draw a uniform hidden direction
$\Theta^{(j)}\in\{\pm1\}^D$. Conditional on it, draw a fixed table
$S^{(j)}\in\{\pm1\}^{n\times D}$, independently across columns, with column $\ell$ uniform over sign vectors of sum $s\Theta^{(j)}_\ell$. Thus
\[
 \frac1n\sum_{i=1}^n S_{i:}^{(j)}=\rho\Theta^{(j)}
\]
exactly, not just in expectation. Define
\[
 q_j(z)=q_{\Theta^{(j)}/\sqrt D}(z),\qquad
 \widetilde q_{ij}(z)=\rho^{-1}q_{S_{i:}^{(j)}/\sqrt D}(z),
\]
so that $n^{-1}\sum_i\widetilde q_{ij}=q_j$. With blocks $x=(x_1,\ldots,x_T)$, let
\begin{equation}\label{eq:Ri}
 R_i(x)=-\widetilde q_{i1}(x_1)
 -\sum_{j=2}^T G_p(q_{j-1}(x_{j-1}))
                 (\widetilde q_{ij}(x_j)+1)
 +\sum_{j=1}^T\norm{x_j}^2.
\end{equation}
Its average is
\begin{equation}\label{eq:R}
 R(x)=-q_1(x_1)-\sum_{j=2}^T G_p(q_{j-1}(x_{j-1}))(q_j(x_j)+1)
 +\sum_{j=1}^T\norm{x_j}^2.
\end{equation}
All component functions are $C^{p+1}$ and are bounded below. The total dimension is $d=TD$.

\begin{lemma}[Gap, gradient barrier, and individual regularity]\label[lemma]{lem:base}
For the above family,
\begin{align}
 R(0)-\inf R&\le2T,\label{eq:gap}\\
 \min_j q_j(x_j)<b&\quad\Longrightarrow\quad
 \norm{\nabla R(x)}>\frac13,\label{eq:barrier}\\
 \Lip(\nabla^pR_i)&\le\frac{H_{p+1}}\rho\quad(i\in[n]).\label{eq:Lip}
\end{align}
In particular, the regularity constant has no $T$ or $D$ factor.
\end{lemma}
\begin{proof}
Since $|q_j|<1$ and $0\le G_p\le1$, $R(0)=0$ and $R\ge-2T$, proving \eqref{eq:gap}.

For the gradient barrier, we first prove that for every unit $u$, $c\ge1$, and $q=q_u(x)\le b=1/4$,
\begin{equation}\label{eq:radial-barrier}
 \norm{2x-c\nabla q_u(x)}>\frac13.
\end{equation}
Let $\zeta=(1+\norm x^2)^{-1}$. Direct differentiation gives
\[
 \ip{x}{\nabla q_u(x)}=q\zeta,\qquad
 \norm{\nabla q_u(x)}^2=\zeta(1-q^2-q^2\zeta),\qquad
 0<\zeta\le1-q^2.
\]
Consequently
\[
 \norm{2x-c\nabla q_u(x)}^2
 =4(\zeta^{-1}-1)-4cq\zeta+c^2\zeta(1-q^2-q^2\zeta).
\]
If $0\le q\le1/4$, its unconstrained minimizer in $c$ is
\[
 \frac{2q}{1-q^2-q^2\zeta}
 \le\frac{2q}{(1-q^2)^2}<1.
\]
Hence $c=1$ minimizes over $c\ge1$. The derivative of the resulting expression with respect to $\zeta$ is at most $-4+1=-3$, so its minimum is at $\zeta=1-q^2$. The square root there is
\[
 \frac{(1-q^2)^2-2q}{\sqrt{1-q^2}}
 \ge (1-b^2)^2-2b=\frac{97}{256}>\frac13.
\]
If $q\le0$, $c=1$ again minimizes. The negative cross term is now nonnegative, and the smallest singular value of $D\Psi(x)$ is $(1+\norm x^2)^{-3/2}$. Thus the squared norm is at least
\[
 4\norm x^2+(1+\norm x^2)^{-3}\ge1,
\]
using $(1+t)^{-3}\ge1-3t$. This proves \eqref{eq:radial-barrier}.

Let $j$ be the first index with $q_j(x_j)<b$. Its incoming coefficient is one, while its outgoing contribution gives
\[
 \nabla_{x_j}R(x)=2x_j-c_j\nabla q_j(x_j),\qquad
 c_j=1+\ind_{\{j<T\}}G_p'(q_j(x_j))(q_{j+1}(x_{j+1})+1)\ge1.
\]
Apply \eqref{eq:radial-barrier}.

Finally, put $k=p+1\ge3$. For one link and arbitrary global directions $h^1,\ldots,h^k$, the product rule and \eqref{eq:composition-constants} imply
\begin{equation}\label{eq:link-bound}
 |D^k[ G_p(q_{j-1})(\widetilde q_{ij}+1)]
      [h^1,\ldots,h^k]|
 \le\frac{C_{p,k}}\rho
 \prod_{\ell=1}^k(\norm{h^\ell_{j-1}}+\norm{h^\ell_j}).
\end{equation}
Indeed, a derivative of the weak factor of positive order $r$ is bounded by $B_r/\rho$, and its order-zero value plus one is bounded by $2B_0/\rho$. Each subset in the product rule has coefficient bounded by $C_{p,k}/\rho$, and summing its block products gives the product on the right.

For $a_j^\ell=\norm{h^\ell_{j-1}}+\norm{h^\ell_j}$, bounded overlap gives
\[
 \norm{a^\ell}_k\le\norm{a^\ell}_2\le2\norm{h^\ell}.
\]
H\"older's inequality therefore yields
\[
 \sum_{j=2}^T\prod_{\ell=1}^k a_j^\ell
 \le 2^k\prod_{\ell=1}^k\norm{h^\ell}.
\]
The initial term contributes $B_k/\rho$, and the quadratic contributes zero. Consequently,
\[
 \sup_x\opnorm{D^{p+1}R_i(x)}\le H_{p+1}/\rho.
\]
All these derivatives are continuous, including at the gate interfaces by the extra flatness. Integrating along the segment from $x$ to $y$ proves \eqref{eq:Lip}.
\end{proof}
\begin{remark}
Bounded-overlap derivative estimates already occur in higher-order zero-chain constructions \citep{carmon2020,emmenegger2022}. Here the estimate is applied to radial blocks with a weak encoding: only one factor of each link carries $1/\rho$, and the blockwise H\"older argument preserves that cost uniformly in both $T$ and $D$.
\end{remark}

\section{\texorpdfstring{Exact-$p$-jet hiding}{Exact-p-jet hiding} and the row information bound}\label{sec:info}
The analytic construction is now admissible, but its gradient barrier is useful only if an algorithm cannot open many stages quickly. The argument has four steps: bound the information in a single table; establish stage freshness before its first row is exposed; censor replies after a stage\textquotesingle s first alignment; and couple the public truncated transcript to the complete oracle. Extra disclosures are used only for information analysis, and the public algorithm ignores them. Future-stage independence is established without conditioning on the absence of earlier future alignments. The tables are fixed before any query, so repeated component queries at one stage reveal the same row, not fresh samples.

\subsection{A censored one-table experiment}
For one independent pair $(\Theta,S)$ as above, let $W$ be arbitrary independent side information including the learner's random tape. The learner may receive one free initialization row whose index is chosen using $W$ before the row is seen. At each of at most $B_*$ rounds, it chooses an adaptive index $I_t$ and an adaptive vector $V_t$ of norm at most one, then observes row $S_{I_t:}$. It may also receive any deterministic function of the query and previously revealed rows. It receives no alignment flag and no hidden direction. Let $K_t$ be the number of distinct rows known before round $t$, including the free initialization row. Define
\[
 r_\rho=\left\lfloor\frac1{4096}\min\{n,\rho^{-2}\}\right\rfloor,
 \qquad
 \mathcal E_r=\left\{\exists t\le B_*:\ K_t<r,\
 D^{-1/2}\ip{\Theta}{V_t}>a\right\}.
\]

\begin{lemma}[Censored adaptive row information]\label[lemma]{lem:row-info}
Assume $\rho\le1/8$ and $1\le r\le n/4096$. Let $\mathcal T_r$ be the transcript stopped immediately after the $r$th distinct row is revealed, or at the horizon if fewer are revealed. Then, conditionally on almost every $W=w$,
\begin{equation}\label{eq:information}
 I(\Theta;\mathcal T_r\mid W=w)\le Dr\rho^2.
\end{equation}
If $r=r_\rho\ge2$ and
\begin{equation}\label{eq:one-stage-D}
 D\ge2048+512\log(B_*+2),
\end{equation}
then $\Prob(\mathcal E_r\mid W=w)<1/8$.
\end{lemma}
\begin{proof}
Condition on $W=w$ and pad the stopped transcript to fixed length by cemetery symbols. Query choices contribute no conditional information because they are functions of preceding observations. If the free initialization row is present, treat it as a zero-th exposure and include it in the transcript; it is charged to the row-information budget, though not to the stage's query budget.

Suppose $d<r$ distinct rows have been seen and their sum in column $\ell$ is $\Sigma_{d,\ell}$. For a previously unqueried index, exchangeability gives
\[
 \Prob(S_{I_t\ell}=1\mid\Theta_\ell=\theta,\text{past})
 =\frac{n+s\theta-d-\Sigma_{d,\ell}}{2(n-d)}=:p_\theta.
\]
The two parameters differ by $s/(n-d)$ and lie in $[1/4,3/4]$, since $s\le n/8$ and $d\le n/4096$. On this interval the negative second derivative of binary entropy is at most $16/3$. For any posterior probability $\pi=\Prob(\Theta_\ell=1\mid\text{retained history})$, the second-order Jensen gap is at most $\tfrac12(16/3)\pi(1-\pi)(p_+-p_-)^2$, hence at most
\[
 \frac{2}{3}\left(\frac{s}{n-d}\right)^2\le\rho^2.
\]
Expose the new row coordinate by coordinate. Given the hidden vector and the revealed rows, the remaining columns are conditionally independent. To justify this under adaptive row selection, fix a possible retained history: each selected index is then fixed, and consistency with the policy imposes no additional constraint beyond the cells already observed. The likelihood of these cells factors across columns conditional on $\Theta$. When revealing coordinates of the new row successively, the channel to its next cell consequently depends on the hidden vector only through $\Theta_\ell$. The entropy bound above is uniform in its posterior probability; it does not assume that hidden coordinates are posterior-independent. Thus the preceding information bound applies to each of the $D$ new cells. Repeated rows contribute no information. The conditional chain rule gives \eqref{eq:information}, including the optional first row.

Now let $r=r_\rho$. The information is at most $D/4096$. Under the product law consisting of the uniform hidden vector and an independent transcript with its actual marginal distribution, every retained $V_t$ is independent of $\Theta$. A Rademacher tail bound and a union bound give
\[
 q:=\Prob_{\rm product}(\mathcal E_r)
 \le B_*\exp(-a^2D/2)=B_*e^{-D/128}.
\]
Writing $u$ for the true probability of $\mathcal E_r$, relative-entropy data processing yields
\[
 I(\Theta;\mathcal T_r\mid W=w)
 \ge u\log(1/q)-\log2.
\]
Condition \eqref{eq:one-stage-D} gives
$\log(1/q)\ge3D/512$ and, because $D\ge2048$,
$D/4096+\log2<3D/4096$. Hence
\[
 u\le\frac{D/4096+\log2}{D/128-\log B_*}<\frac18.
\]
The event is measurable from the retained transcript and $\Theta$: the query selecting the $r$th row is retained before stopping. This is why the stopping convention is specified explicitly.
\end{proof}

\subsection{A stronger truncated simulator and first-hit censoring}
For $0\le h\le T$, define $R_i^{[h]}$ by retaining the first term in \eqref{eq:Ri}, the quadratic, and only links with $2\le j\le\min\{h+1,T\}$. The completed prefix initially has length $h=0$; the current stage is $j=h+1$. At a query point $x$, the simulator checks the current alignment before giving its reply. If $q_j(x_j)>a$, it promotes exactly one stage, to $h+1$; otherwise it leaves $h$ unchanged. Its public reply is the exact $p$-jet of the truncated function at the updated prefix. Once $h=T$, keep the prefix fixed and return the full component $p$-jet on every remaining call; no further completion tests or disclosures are made.

For information analysis only, enrich this simulator by revealing the selected row of the current table. On a completion response it additionally reveals the entire data of the completed stage and the selected row of the newly current stage. The latter is its one free initialization row. These disclosures make the simulator more informative, not less.

\begin{lemma}[Exact jet truncation and reply factorization]\label[lemma]{lem:jet}
If every gate of a link omitted from $R_i^{[h]}$ has input at most $a$, then its full $p$-jet is zero and
\[
 \cJ^pR_i(x)=\cJ^pR_i^{[h]}(x).
\]
Before completion of the current stage $j$, the public truncated reply is a deterministic function of the completed-stage data, the query, and row $S^{(j)}_{i:}$. At completion it is a function of these data, $\Theta^{(j)}$, and row $S^{(j+1)}_{i:}$ if $j<T$.
\end{lemma}
\begin{proof}
An omitted link has form $G_p(q(x))(\widetilde q_i(y)+1)$. Every derivative through order $p$ is a sum of product-rule and chain-rule terms containing some $G_p^{(k)}(q(x))$ with $0\le k\le p$. All such factors vanish for $q(x)\le a$, including equality. For an active incoming link, every derivative of $\widetilde q_i$ is determined by its selected row. The outgoing link enters only on completion, when the current direction may be revealed. No later hidden table or direction enters the truncated reply.
\end{proof}

\begin{lemma}[Prefix measurability and stage freshness]\label[lemma]{lem:prefix-freshness}
Let $\xi$ be the algorithm's private random tape, let
$\mathcal Z_j=(\Theta^{(j)},S^{(j)})$, and let $h_t$ be the simulator's prefix length after response $t$.
For $g<T$, the event $\{h_t=g\}$ and the enriched history on that event are measurable with respect to
\[
 \mathcal A_{g+1}=\sigma(\xi,\mathcal Z_1,\ldots,\mathcal Z_{g+1}).
\]
For $g=T$, $\mathcal A_T$ suffices. In particular, conditional on the enriched history and $h_t=g<T$, stages strictly beyond $g+1$ have their independent prior distributions.
Before the initialization row of a reached stage $j$ is exposed, its table also has its independent prior conditional on the entire earlier-stage data and $\xi$.
\end{lemma}
\begin{proof}
Induct on $t$. On a branch with pre-query prefix $g$, the query and completion test use only $\mathcal A_{g+1}$. A noncompletion reply uses earlier data and one row of stage $g+1$. A completion reply may reveal the complete data of stage $g+1$ and one row of stage $g+2$, but nothing later. The event $\{h_t=g\}$ is the disjoint union of a noncompletion branch from prefix $g$ and a completion branch from prefix $g-1$; both have measurable representations in $\mathcal A_{g+1}$. This proves the first statement, including event measurability, rather than only measurability conditional on a possibly informative selection event. Independence of $\xi,\mathcal Z_1,\ldots,\mathcal Z_T$ gives the future-stage assertion.

For the last assertion, the query completing stage $j-1$, its completion test, and the selection of its component index are functions of $\xi,\mathcal Z_1,\ldots,\mathcal Z_{j-1}$ before the response is generated. Order that response conceptually by revealing completed-stage data first and the stage-$j$ initialization row second. Reaching this intermediate instant involves no stage-$j$ data. Conditional on the earlier variables, the new table retains its prior and its initialization-row index is fixed. This conceptual ordering changes no observable response available before the next query.
\end{proof}

Fix a reached stage $j$ and condition at the intermediate instant in \Cref{lem:prefix-freshness}. Independent information may be supplied without charge in the one-table experiment, so conditioning on the complete earlier tables is legitimate. No conditioning on the absence of a future-alignment event is performed.

Couple the real enriched stage to a counterfactual experiment that always uses the noncompletion reply map and always reports noncompletion, even after a first alignment would have occurred. This is a censored row learner of the preceding type: its reply maps depend only on rows and the independent old data. Up to the first completion query, both interactions have the same pre-query history and choose the same query. In particular,
\begin{equation}\label{eq:first-hit}
 \{\text{completion with fewer than }r_\rho\text{ previously exposed rows}\}
 \subseteq\mathcal E_{r_\rho}.
\end{equation}
The completion reply, and any direction it reveals, occurs after the event in \eqref{eq:first-hit}; it is not included in the information transcript. Noncompletion flags are forced constants in the censored experiment and contribute no information there.

\begin{corollary}[Each reached stage is rarely short]\label[corollary]{cor:short}
Under \eqref{eq:one-stage-D}, a reached stage completes in fewer than $r_\rho$ calls assigned to it with conditional probability at most $1/8$, before its initialization row is exposed. An unreached stage is declared not short.
\end{corollary}
\begin{proof}
On its $q$th assigned call with $q<r_\rho$, it has seen at most its initialization row and $q-1$ previous response rows, so the number of distinct pre-query rows is at most $q<r_\rho$. Apply \eqref{eq:first-hit} and \Cref{lem:row-info}. Conditioning on all earlier data is legitimate because they are independent of the new table. Averaging over that larger conditioning field gives the same bound conditional on the actually observed pre-initialization history.
\end{proof}

\subsection{Adding stage costs and coupling to the true oracle}
Assign each call to its pre-query current stage. If $N\le r_\rho J/2$ and $J$ stages complete, at least $J/2$ of those stages must be short, since every other stage consumes at least $r_\rho$ disjoint assigned calls. By \Cref{cor:short}, the expected number of short stages among the first $J$ is at most $J/8$. Markov's inequality gives
\begin{equation}\label{eq:stagecount}
 \Prob(h_N\ge J)\le\frac14.
\end{equation}
No independence of the stage durations is assumed.

Let $\mathcal B$ be the event that at some queried point a block strictly beyond the pre-query current stage has alignment greater than $a$, also including blocks strictly beyond the final current stage at the unqueried output. At every history of the enriched truncated simulator, all such strictly future stages retain their independent prior. This is precisely the conditional independence in \Cref{lem:prefix-freshness}; it is valid even when prefix lengths and component indices are adaptive. Therefore, regardless of query radii,
\begin{equation}\label{eq:future}
 \Prob(\mathcal B)\le(N+1)T e^{-D/128},
\end{equation}
because $\norm{\Psi(x)}<1$. This argument is in the truncated simulator and does not condition on the absence of future alignments.

To relate it to a true exact-oracle algorithm, run the same public policy and random tape in the true interaction and in the truncated simulator, ignoring the simulator's extra disclosures. On $\mathcal B^c$, their public transcripts agree inductively. On a noncompletion query all omitted gates, starting with the current outgoing gate, are flat. On a completion query the outgoing link is included, and every still-omitted gate is flat. \Cref{lem:jet} applies in both cases. The public outputs agree as well. We do not assert equality between a public transcript and an enriched transcript; enrichment is used only for domination and information analysis.

\begin{proposition}[Unfinished direction under arbitrary exact-$p$-jet access]\label[proposition]{prop:sequential}
Let $r_\rho\ge2$, $1\le J<T$, and $N\le r_\rho J/2$. If
\begin{equation}\label{eq:D-final}
 D\ge\max\{2048+512\log(N+2),\ 128\log(128(N+1)T)\},
\end{equation}
then for every randomized adaptive true exact-$p$-jet algorithm, with probability at least $95/128$ its output has $q_k(\widehat x_k)\le a$ in some block. The same conclusion holds after arbitrary positive input scaling.
\end{proposition}
\begin{proof}
Equations \eqref{eq:stagecount} and \eqref{eq:future} imply probability at least $1-1/4-1/128=95/128$ for $h_N<J$ and $\mathcal B^c$ in the truncated run. Since $J<T$, a block strictly beyond its final current stage remains; its output alignment is at most $a$. Transfer the event to the true public interaction by the preceding coupling. For input scaling $x\mapsto\beta x$, replace every alignment vector by $\Psi(\beta x)$, which is still bounded in norm by one.
\end{proof}

\section{Scaling and Complete Parameter Coverage}\label{sec:scaling}
Let $H_*=H_{p+1}$. Scale the components by
\begin{equation}\label{eq:scaling}
 f_i(y)=\alpha R_i(\beta y),\qquad
 \beta^p=\frac{L_p\rho}{6H_*\eps},\qquad
 \alpha=\frac{6\eps}{\beta}.
\end{equation}
Then
\begin{align}
 \Lip(\nabla^pf_i)&\le\frac{\alpha\beta^{p+1}H_*}{\rho}=L_p,\\
 F(0)-\inf F&\le2\alpha T,\\
 \min_jq_j(\beta y_j)<b&\quad\Longrightarrow\quad
 \norm{\nabla F(y)}>2\eps.
\end{align}
Choose
\begin{equation}\label{eq:T}
 T=\left\lfloor c_{T,p}\K\rho^{1/p}\right\rfloor,
 \qquad
 c_{T,p}=\frac1{2\cdot6^{1+1/p}H_*^{1/p}}.
\end{equation}
This guarantees the gap constraint exactly. If $c_{T,p}\K\rho^{1/p}\ge4$, reserve the last stage and set $J=T-1$. \Cref{prop:sequential} gives the lower term
\begin{equation}\label{eq:bias-balance}
 \Omega_p\!\left(\K\rho^{1/p}\min\{n,\rho^{-2}\}\right)
 =\Omega_p\!\left(\K\min\{n\rho^{1/p},\rho^{-2+1/p}\}\right).
\end{equation}
The first branch increases in $\rho$ and the second decreases. They balance at $\rho\asymp n^{-1/2}$.

For example, a universal threshold $n_0=2^{16}$ is sufficient for the bias and row-count conditions below. For $n\ge n_0$, choose $s$ with the parity of $n$ in $[\sqrt n,\sqrt n+2]$. Then
\[
 n^{-1/2}\le\rho\le2n^{-1/2},\qquad
 r_\rho\ge n/32768,\qquad
 T=\Theta_p(\K n^{-1/(2p)})
\]
whenever the last quantity is sufficiently large. Hence
\begin{equation}\label{eq:sharp-lower}
 \Omega_p\!\left(\K n^{1-1/(2p)}\right)
\end{equation}
queries are necessary with constant success probability. The probability of an $\eps$-stationary output under the hard distribution is at most $33/128<1/3$ for a sufficiently small multiple of this budget. The tables and directions are fixed before the algorithm runs. Averaging over this distribution and the algorithm's private randomness, then fixing an instance, gives a worst-case deterministic finite sum for each randomized algorithm.

\subsection{\texorpdfstring{The additive $n$ term}{The additive n term}, including arbitrarily small positive gap}
Choose a hidden index $I_*$ uniformly from $[n]$ and an independent hidden scalar $Z$ uniformly from $\{-3,-1,1,3\}$. In dimension one set
\begin{equation}\label{eq:hidden-index}
 f_i(x)=\frac\gamma2x^2-4n\eps Z\ind_{\{i=I_*\}}x,
 \qquad \gamma=\frac{72\eps^2}{\Delta}.
\end{equation}
Then
\[
 F(x)=\frac\gamma2x^2-4\eps Zx,\qquad
 F(0)-\inf F=\frac{Z^2\Delta}{9}\le\Delta.
\]
For every $p\ge2$ the $p$th-derivative Lipschitz constant is zero. Until the hidden index is queried, all replies are identical to those of the public quadratic $\gamma x^2/2$ and reveal nothing about $Z$. Under the all-quadratic transcript, at most $N$ distinct indices are tried, so the probability of hitting $I_*$ is at most $N/n$, even under adaptive index selection. The four possible success intervals
$|\gamma\widehat x-4\eps Z|\le\eps$ are pairwise disjoint. Thus
\[
 \Prob(\text{success})\le\frac{N}{n}+\frac14\le\frac5{16}<\frac13
 \qquad\text{if }N\le n/16.
\]
This proves $\Omega(n)$ for every $\Delta>0$, including when the nontrivial chain cannot fit into the gap budget. If $\K n^{-1/(2p)}$ is below a sufficiently large $p$-dependent constant, $S_p=O_p(n)$, so this single hidden-index family already lower-bounds the whole target expression.

\subsection{Small component counts}
For completeness, the finite set of component counts below the universal large-$n$ threshold does not require an external full-function lower bound. Set $\rho=1$ and take every row in table $j$ equal to $\Theta^{(j)}$, so every component is the same radial-chain function. The analytic estimates and exact-jet coupling remain valid. This case uses only those estimates and the future-alignment bound \eqref{eq:future}, not \Cref{lem:row-info}, whose small-bias hypothesis does not hold at $\rho=1$. At most one stage is promoted per query. If $N\le T-2$, at least one strictly future block remains at the output deterministically in the truncated interaction. Choose $D$ to make \eqref{eq:future} at most $1/128$. Scaling with \eqref{eq:scaling}--\eqref{eq:T} gives an $\Omega_p(\K)$ lower bound when $\K$ is large enough. Together with \eqref{eq:hidden-index}, this is $\Omega_p(n+\K n^{1-1/(2p)})$ for the finitely many remaining $n$, after decreasing the constant. Thus all $n\ge1$ are covered.

\begin{remark}[Dimension of a hard instance]
The constructions can be taken in finite dimension
\[
 d\le C_p\bigl(1+\K n^{-1/(2p)}\bigr)
 \log\bigl(2+n+\K n^{1-1/(2p)}\bigr).
\]
For large $n$ this follows directly from $d=TD$ and \eqref{eq:D-final}; the small-$n$ and hidden-index cases fit after adjusting $C_p$.
\end{remark}

\begin{proof}[Proof of \Cref{thm:main}]
The lower bounds in this section hold in the individual class. \Cref{prop:upper} gives the upper bound in the mean-squared class. Since \eqref{eq:individual} implies \eqref{eq:ms}, the sandwich in \eqref{eq:main} follows.

If $\Delta=0$, the initial point is a global minimizer of a differentiable function and is stationary. If $L_p=0$, every component is a polynomial of degree at most $p$, recovered in full by one $p$-jet query at the origin. Hence $n$ calls determine $F$ exactly. To see directly that an $\eps$-stationary point exists, let $\eta=\eps^2/(8\Delta)$ and minimize $F(x)+\eta\norm{x}^2/2$. This function is coercive because $F$ is bounded below. At a global minimizer $x_\eta$,
\[
 \nabla F(x_\eta)=-\eta x_\eta,\qquad
 \frac\eta2\norm{x_\eta}^2\le F(0)-\inf F\le\Delta,
\]
so $\norm{\nabla F(x_\eta)}\le\eps/2$. By continuity, a rational point with gradient norm below $\eps$ exists. For the known polynomial choose the first such point in a fixed enumeration of $\mathbb Q^d$. This is a measurable zero-additional-query computation. The hidden-index quadratic family proves the matching $\Omega(n)$ for $\Delta>0$.
\end{proof}

\section{Additional Technical Details}\label{sec:technical}
\subsection{Measurable model minimization}\label{sec:selector}
\begin{lemma}[A Borel model selector]\label[lemma]{lem:selector}
For fixed $r>0$, there is a Borel map from the finite-dimensional coefficient vector of \eqref{eq:localmodel} to a global minimizer on $\overline B(0,r)$.
\end{lemma}
\begin{proof}
Write the coefficient vector as $c$ and the continuous objective as $q(c,s)$. Its minimum $v(c)$ over the fixed compact ball is continuous: uniform convergence of $q(c_k,\cdot)$ to $q(c,\cdot)$ on the ball holds whenever $c_k\to c$. Start with a closed cube containing the ball and subdivide it into finitely many closed cubes of half the side length. Select the first subcube whose intersection with the ball contains a minimizer, and repeat within that selected cube. For a fixed subcube $Q$, this selection test is Borel because
\[
 \min_{s\in Q\cap\overline B(0,r)}q(c,s)=v(c)
\]
is an equality of continuous functions when the intersection is nonempty; an empty intersection is always rejected. At each depth, there are finitely many possible cube histories, so every selection is Borel. The nested nonempty compact intersections with the minimizer set have diameters tending to zero. Their unique limit is a minimizer and a Borel function of $c$. This selector is an admissible exact-oracle decision rule.
\end{proof}

\subsection{Growth of the order-dependent constants}\label{sec:constant-growth}
Let $k\le p+1$. The envelope in \eqref{eq:gate-constants} obeys
\[
 E_{p,k}\le 8^k Z_p\,2^{p+1}(2p+2)^{k-1}
 \quad(k\ge1),\qquad
 Z_p=(2p+3)\binom{2p+2}{p+1}\le(2p+3)2^{2p+2}.
\]
For a partition $\pi$ of $[k]$ into $h$ nonempty blocks,
\[
 \prod_{V\in\pi} B_{|V|}
 \le 2^h4^k k^k,
\]
since $\prod_{V\in\pi}|V|^{|V|}\le k^k$. There are at most $k^k$ partitions. Substitution into \eqref{eq:composition-constants}--\eqref{eq:Hconstant} gives
\[
 \log(2+H_{p+1})\le C(p+1)\log(p+2)
\]
for a numerical $C$. These are explicit, dimension-free bounds; they are not optimized in $p$.

\subsection{Strictness of the mean-squared assumption}\label{sec:strictness}
For $n\ge2$, the inclusion from individual to mean-squared smoothness is strict even among bounded-below one-dimensional examples with arbitrarily small positive initial-gap allowance. For $\omega>0$, set
\[
 f_1(x)=\sqrt n\,L_p\omega^{-(p+1)}\sin(\omega x),
 \qquad f_i(x)=0\quad(i>1).
\]
The Lipschitz constant of $\nabla^p f_1$ is exactly $\sqrt n L_p$, while the normalized mean-squared increment is bounded by $L_p|x-y|$. Thus \eqref{eq:ms} holds but \eqref{eq:individual} with the same parameter does not. Moreover,
$F(0)-\inf F=L_p n^{-1/2}\omega^{-(p+1)}$, which is at most any prescribed $\Delta>0$ for sufficiently large $\omega$.

\subsection{From a hard distribution to a deterministic instance}
All hidden signs, tables, and the hidden-index choices are sampled once before interaction. For every fixed randomized decision rule with the stated budget, the preceding lower-bound proofs bound its success probability averaged jointly over these data and its private tape. If every deterministic realization had success probability at least $2/3$ over that tape, the joint average would also be at least $2/3$. The proved upper bound on joint success is smaller. Hence a fixed realization defeats the randomized rule. No fresh stochastic oracle, adaptive change of the underlying function, or restriction to deterministic algorithms is used.
\end{document}